\documentclass[a4paper,12pt]{amsart}
\usepackage{environ}
\usepackage{fancyhdr}
\usepackage{mabliautoref}
\usepackage{amssymb,amsthm,amsmath}
\RequirePackage[dvipsnames,usenames]{xcolor}
\usepackage{hyperref}
\usepackage{cleveref}
\usepackage{mathtools}
\usepackage{tcolorbox}
\usepackage[all]{xy}
\usepackage{tikz}
\usepackage{tikz-cd}
\usetikzlibrary{decorations.markings,shapes,positioning}
\usepackage{enumitem}
\usepackage{stmaryrd}
\usepackage{xfrac}
\usepackage{fix-cm}
\usepackage{faktor}
\usepackage{minibox}
\usepackage{commath}
\usepackage{chngcntr}
\usepackage{setspace}
\usepackage{array}
\usepackage[toc,page]{appendix}
\usepackage{calligra,mathrsfs}
\usepackage{mathtools}
\usepackage{bbm}

\hypersetup{
	bookmarks,
	bookmarksdepth=3,
	bookmarksopen,
	bookmarksnumbered,
	pdfstartview=FitH,
	colorlinks,backref,hyperindex,
	linkcolor=Sepia,
	anchorcolor=BurntOrange,
	citecolor=Cyan,
	filecolor=BlueViolet,
	menucolor=Yellow,
	urlcolor=OliveGreen
}

\newtcolorbox{activitybox}[1][]{%
	breakable,
	enhanced,
	colback=lightergray,
	boxrule=3pt,
	arc=5pt,
	outer arc=5pt,
	boxsep=10pt,
	colframe=darkergray,
	coltitle=white,
	#1
}

\newcommand{\quot}[2]{%
	\raise1ex\hbox{$#1$}\Big/\lower1ex\hbox{$#2$}%
}

\newcommand{\colim}{\varinjlim}
\renewcommand{\lim}{\varprojlim}

\makeatletter\def\Rvarlim@#1#2{%
	\vtop{\m@th\ialign{##\cr
			\hfil$#1\operator@font Rlim$\hfil\cr
			\noalign{\nointerlineskip\kern1.5\ex@}#2\cr
			\noalign{\nointerlineskip\kern-\ex@}\cr}}%
}
\makeatother

\makeatletter\def\Rlim{%
	\mathop{\mathpalette\Rvarlim@{\leftarrowfill@\textstyle}}\nmlimits@
}
\makeatother

\DeclareMathOperator{\GR}{GR}

\newcommand{\expl}[2]{\underset{\mathclap{\minibox[c]{$\uparrow$\\ \fbox{\footnotesize #2}}}}{#1}}

\newcommand{\longexpl}[2]{\underset{\mathclap{\minibox[c]{$\big\uparrow$\\[3 pt] \fbox{\footnotesize #2}}}}{#1}}

\newcommand{\dra}{\dashrightarrow}

\def\xto#1#2{\xrightarrow{#1}{#2}}
\def\xto#1{\xrightarrow{#1}}
\newcommand{\surj}{\twoheadrightarrow}

\newcommand{\ttilde}{\widetilde}

\newcommand{\HHom}{\mathcal{H}om}

\newcommand{\stacksproj}[1]{{\cite[Tag~\href{http://stacks.math.columbia.edu/tag/#1}{#1}]{Stacks_Project}}}

\newcommand{\inc}{\subseteq}

\newcommand{\esp}{\mbox{ }}
\newcommand{\bighat}{\widehat}

\DeclareMathAlphabet{\mathchanc}{OT1}{pzc}%
{m}{it}

\newcommand{\bigset}[2]{\big\{ \ #1 \  \big| \ #2 \ \big\}}

\newcommand{\bF}{\mathbb{F}}

\newcommand{\bQ}{\mathbb{Q}}

\newcommand{\bZ}{\mathbb{Z}}

\newcommand{\scr}{\mathcal}

\newcommand{\cF}{\scr{F}}

\newcommand{\cH}{\scr{H}}

\newcommand{\cL}{\scr{L}}
\newcommand{\cM}{\scr{M}}
\newcommand{\cN}{\scr{N}}
\newcommand{\cO}{\scr{O}}
\newcommand{\cP}{\scr{P}}

\newcommand{\cR}{\scr{R}}

\DeclareMathOperator{\injj}{{inj}}

\DeclareMathOperator{\FM}{{FM}}

\DeclareMathOperator{\alb}{{alb}}
\DeclareMathOperator{\Alb}{{Alb}}

\DeclareMathOperator{\coker}{{coker}}

\DeclareMathOperator{\Frac}{Frac}

\DeclareMathOperator{\Tor}{Tor}

\DeclareMathOperator{\id}{{id}}

\DeclareMathOperator{\im}{{im}}

\DeclareMathOperator{\length}{{length}}

\DeclareMathOperator{\Pic}{Pic}

\DeclareMathOperator{\GL}{{GL}}

\DeclareMathOperator{\Spec}{{Spec}}

\DeclareMathOperator{\Supp}{{Supp}}

\DeclareMathOperator{\Crys}{Crys}
\DeclareMathOperator{\coh}{coh}

\DeclareMathOperator{\RGamma}{R\Gamma}

\DeclareMathOperator{\perf}{perf}

\DeclareMathOperator{\Coh}{Coh}

\DeclareMathOperator{\ShHom}{\mathscr{H}\text{\kern -3pt {\calligra\large om}}\,}

\DeclareFontFamily{OT1}{pzc}{}
\DeclareFontShape{OT1}{pzc}{m}{it}{<-> s * [1.200] pzcmi7t}{}
\DeclareMathAlphabet{\mathpzc}{OT1}{pzc}{m}{it}

\renewcommand{\phi}{\varphi}

\newcommand{\factor}[2]{\left. \raise 2pt\hbox{\ensuremath{#1}} \right/
	\hskip -2pt\raise -2pt\hbox{\ensuremath{#2}}}

\makeatletter
\renewcommand\subsection{
	\renewcommand{\sfdefault}{pag}
	\@startsection{subsection}%
	{2}{0pt}{.8\baselineskip}{.4\baselineskip}{\raggedright
		\sffamily\itshape\small\bfseries
}}
\renewcommand\section{
	\renewcommand{\sfdefault}{phv}
	\@startsection{section} %
	{1}{0pt}{\baselineskip}{.8\baselineskip}{\centering
		\sffamily
		\scshape
		\bfseries
}}
\makeatother

\definecolor{gr}{rgb}{0,0.5,0}

\newcommand{\fm}{\mathfrak{m}}

\newcommand{\Addresses}{{
		\bigskip
		\footnotesize
		
		\textsc{\'Ecole Polytechnique F\'ed\'erale de Lausanne, SB MATH CAG, MA C3 615 (B\^atiment MA), Station 8, CH-1015 Lausanne, Switzerland}\par\nopagebreak
		\textit{E-mail address}: \texttt{jefferson.baudin@epfl.ch}

}}

\author{Jefferson Baudin}
\date{}

\usepackage[margin=1.0in]{geometry}
\setlist{  
	listparindent=\parindent,
	parsep=0pt,
}

\subjclass[2020]{14K05, 14G17, 14F17}
\keywords{Generic vanishing, positive characteristic, Cartier crystals, Grauert--Riemenschneider vanishing}

\title[Euler characteristic of weakly ordinary irregular varieties]{On the Euler characteristic of weakly ordinary varieties of maximal Albanese dimension}
\begin{document}
	\maketitle
	\begin{abstract}
		We show that a smooth proper weakly ordinary variety $X$ of maximal Albanese dimension satisfies $\chi(X, \omega_X) \geq 0$. We also show that if $X$ is not of general type, then $\chi(X, \omega_X) = 0$ and the Albanese image of $X$ is fibered by abelian varieties. The proof uses the positive characteristic generic vanishing theory developed by Hacon-Patakfalvi, as well as our recent Witt vector version of Grauert--Riemenschneider vanishing.
	\end{abstract}	
	
	\tableofcontents
\section{Introduction}

\subsection{Main results}

A classical theorem of Green and Lazarsfeld (\cite{Green_Lazarsfeld_Generic_vanishing}) is that is a smooth projective complex variety $X$ has maximal Albanese dimension (i.e. $\dim(X) = \dim(\alb_X(X))$), then $\chi(X, \omega_X) \geq 0$.

The proof of this result is based on generic vanishing theory. Since the emergence of a positive characteristic version of generic vanishing \cite{Hacon_Pat_GV_Characterization_Ordinary_AV, Hacon_Pat_GV_Geom_Theta_Divs, Baudin_Positive_characteristic_generic_vanishing_theory}, one might hope to obtain an analogue of this statement in positive characteristic. We prove the following:

\begin{thm_letter}[{\autoref{main_thm_ordinary}}]\label{intro_main_thm}
	Let $X$ be a smooth proper variety of maximal Albanese dimension, and assume that $X$ is weakly ordinary. Then we have \[ \chi(X, \omega_X) \geq 0. \]
\end{thm_letter}

A variety $X$ in positive characteristic is called \emph{weakly ordinary} if for all $i \geq 0$, the map $H^i(X, \cO_X) \to H^i(X, F_*\cO_X)$ induced by taking $p$-powers is an isomorphism. For example, an abelian variety is weakly ordinary if and only if it is ordinary, and hence this property holds for a general abelian variety. \\

It turns out that we actually prove a more general statement, from which \autoref{intro_main_thm} can be deduced:

\begin{thm_letter}[{\autoref{main_thm_ss}}]\label{intro_main_thm_ss}
	Let $X$ be a smooth proper variety admitting a generically finite map to an ordinary abelian variety. Then we have \[ \chi_{ss}(X, \omega_X) \geq 0. \]
\end{thm_letter}

The symbol $\chi_{ss}$ denotes the semistable Euler characteristic, which ignores nilpotence of the Frobenius action on the cohomology groups. In particular, if $X$ is weakly ordinary, then $\chi_{ss}(X, \omega_X) = \chi(X, \omega_X)$. Note that if $A$ is a supersingular abelian variety, then $\chi_{ss}(A, \omega_A) = (-1)^{\dim(A)}$ (which can be negative), so one must impose some ordinarity condition for \autoref{intro_main_thm_ss} to hold. \\

Over the complex numbers, we have in addition the following two implications for smooth varieties of maximal Albanese dimension:

\begin{itemize}
	\item if $X$ is not of general type, then $\chi(X, \omega_X) = 0$;
	\item if $\chi(X, \omega_X) = 0$, then $\alb_X(X)$ is fibered by abelian varieties (i.e. it admits a fibration whose fibers are abelian varieties).
\end{itemize}

We prove the exact analogues in this positive characteristic setting:

\begin{thm_letter}[\autoref{main_thm_ordinary}]\label{intro_main_thm_not_gen_type}
	Let $X$ be a smooth proper weakly ordinary variety of maximal Albanese dimension. Assume further that $X$ is not of general type. Then \[ \chi(X, \omega_X) = 0. \]
\end{thm_letter}

\begin{thm_letter}[{\autoref{euler_char_zero_implies_alb_fibered_by_ab_vars}}]\label{intro_fibered_by_tori}
	Let $X$ be a smooth proper weakly ordinary variety of maximal Albanese dimension, and assume that $\chi(X, \omega_X) = 0$ (e.g. $X$ is not of general type). Then $\alb_X(X)$ is fibered by ordinary abelian varietes.
\end{thm_letter}


%
%
%

As an interesting consequence of (the proof of) \autoref{intro_main_thm_not_gen_type}, we obtain that if $X$ is a smooth proper threefold not of general type that admits a generically finite morphism $a \colon X \to A$ to an ordinary abelian variety, then Grauert--Riemenschneider holds up to nilpotence for $a$, namely all sheaves $R^ia_*\omega_X$ are nilpotent under the Cartier operator (see \autoref{suprising_GR_vanishing}). This fails if we drop the non--general type assumption, or if we work in higher dimension (see \autoref{example_non_GR_van}).

\begin{questions_intro}
	\begin{enumerate}
		\item In the setup of the above paragraph, assume furthermore that $X$ is weakly ordinary. Do we have that $R^ia_*\omega_X = 0$ for all $i > 0$ and that $a_*\omega_X$ is a GV--sheaf?
		\item In general, what happens when we drop any ordinarity condition? For example, does there exist a smooth projective threefold $X$ of maximal Albanese dimension in positive characteristic for which $\chi(X, \omega_X) < 0$? Can this happen when $\Alb(X)$ is ordinary?
	\end{enumerate}
\end{questions_intro}

\subsection{Strategy of proofs}

\setcounter{theorem}{4}

Let us first sketch the proof of Ein and Lazarsfeld's theorem over the complex numbers. To lighten notations, let $a \colon X \to \Alb(X)$ denote the Albanese morphism of $X$. Pick a general numerically trivial line bundle $L$ on $\Alb(X)$, so that by the generic vanishing theorem (\cite[Corollary 4.2]{Hacon_A_derived_category_approach_to_generic_vanishing}), we have \[ H^i(\Alb(X), a_*\omega_X \otimes L) = 0 \] for all $i > 0$. By Grauert--Riemenschneider vanishing (i.e. $R^ja_*\omega_X = 0$ for all $j > 0$), we have \[ H^i(X, \omega_X \otimes a^*L) = H^i(\Alb(X), a_*\omega_X \otimes L) = 0 \] for all $i > 0$, so \[ \chi (X, \omega_X \otimes L) = h^0(X, \omega_X \otimes L) \geq 0. \] Since $\chi(X, \omega_X) = \chi(X, \omega_X \otimes L)$, the proof is complete. \\

We therefore see that there are two main parts in this proof, namely generic vanishing and Grauert--Riemenschneider vanishing. Let us explain their positive characteristic analogues.

\subsubsection{Positive characteristic generic vanishing theory}

Although generic vanishing is known to fail in positive characteristic \cite{Filipazzi_GV_fails_in_pos_char, Hacon_Kovacs_GV_fails_in_pos_char}, a breakthrough of Hacon and Patakfalvi was to find a meaningful analogue that actually holds \cite{Hacon_Pat_GV_Characterization_Ordinary_AV, Hacon_Pat_GV_Geom_Theta_Divs}. Instead of just at certain sheaves (e.g. $a_*\omega_X$), we should remember that they carry certain Frobenius actions (e.g. the Cartier operator), and work up to nilpotence under these actions. The right objects to work with in our contexts are \emph{Cartier modules}, i.e. pairs $(\cM, \theta)$ consisting of a coherent sheaf $\cM$ and a morphism $\theta \colon F_*\cM \to \cM$.

In our case, instead of looking at the usual Euler characteristic of $a_*\omega_X$ (or of general Cartier modules), we should consider their \emph{semistable version} (i.e. the version up to nilpotence under the Cartier operator, see \autoref{rem:def_Cartier_mod}). We prove the following general result:

\begin{prop_letter}[{\autoref{Euler_char_Cartier_modules_ordinary}}]\label{intro_prop_euler_char}
	Let $\cM$ be a Cartier module on an ordinary abelian variety $A$. Then $\chi_{ss}(A, \cM) \geq 0$.
\end{prop_letter}

\begin{rem_blank}
	We expect that a more general statement should hold, namely that if $\cM$ is a Cartier module on an abelian variety without simple factors of $p$--rank zero, then $\chi_{ss}(A, \cM) \geq 0$. Since the author was only able to tackle special cases of this more general question, we leave it for future work. On the other hand, the result fails for arbitrary abelian varieties, since for example $\chi_{ss}(E, \omega_E) = -1$ whenever $E$ is supersingular elliptic curve.
\end{rem_blank}

We want to point out that the proof does not use any twist as in characteristic zero, because the semistable Euler characteristic is \emph{not} invariant under twists (see \cite[Remark 3.4.3]{Baudin_Positive_characteristic_generic_vanishing_theory}). The proof uses the positive characteristic version of generic vanishing to transfer this cohomological question about Cartier modules to a commutative algebraic version about other objects with certain arithmetic actions, called \emph{$V$--modules} (see \autoref{def:V-module}). We then prove the $V$--module reformulation of \autoref{intro_prop_euler_char} by hand. We refer the reader to \autoref{section_Euler_char_V_modules} for further details.

\subsubsection{Grauert--Riemenschneider vanishing in positive characteristic} Similarly as above, it is now well--known that Grauert--Riemenschneider vanishing fails in positive characteristic. One can for example take a cone over a variety over which Kodaira vanishing fails (see \cite[Example 3.11]{Hacon_Kovacs_GV_fails_in_pos_char}).

With the same philosophy as before, one would want to say the following: even though GR vanishing is known to fail in positive characteristic (e.g. \cite[Example 3.11]{Hacon_Kovacs_GV_fails_in_pos_char}), does it hold up to nilpotence? In particular, do we have \[ \chi_{ss}(X, \omega_X) = \chi_{ss}(\Alb(X), a_*\omega_X), \] so that we can conclude the proof by \autoref{intro_prop_euler_char}? 

Unfortunately, GR vanishing does not even hold up to nilpotence (\cite{Baudin_Bernasconi_Kawakami_Frobenius_GR_fails}), so there really is an issue here. Nevertheless, we were recently able to prove a Witt vector version of GR vanishing, namely that in our situation, we have \[ R^ia_*W\omega_{X, \bQ} = 0 \] for all $i > 0$ (see \cite{Baudin_Witt_GR_vanishing_and_applications}). With a bit of trickery, this tells us that the right object to work with in our context is not $a_*\omega_X$, but in fact \[ (a_*W\omega_X)/p. \] 

\begin{thm_letter}[{\autoref{weak_GR_sheaf}, \autoref{cohomology_fg_mod_p}}]\label{intro_weak_GR_sheaf}
	Let $\pi \colon X \to Y$ be a proper and generically finite morphism of varieties, with $X$ smooth. Then there exists a naturally induced sub--Cartier crystal $\omega_{\pi, \GR}$ of $\pi_*\omega_X$ such that \[ \lim_e F^e_*\omega_{\pi, \GR} \cong (\pi_*W\omega_X)/p. \] If $X$ and $Y$ are furthermore proper, then \[ \chi_{ss}(X, \omega_X) = \chi_{ss}(Y, \omega_{\pi, \GR}). \]
\end{thm_letter}

We call this sheaf the \emph{weak Grauert--Riemenschneider sheaf} with respect to $\pi$, in analogy with the notion of Grauert--Riemenschneider sheaf in characteristic zero (although $\omega_X$ and $\omega_{\pi, \GR}$ might not have the same cohomology groups as one would expect in characteristic zero, they at least share the same semistable Euler characteristic). We believe that this statement is of independent interest. 

\begin{rem_blank}
	As a consequence of \cite[Theorem 5.2.4]{Baudin_Witt_GR_vanishing_and_applications} and \cite[Theorem 1.5]{Baudin_Kawakami_Roesler_On_GR_for_klt_CM_schemes}, it holds that if $\pi \colon X \to Y$ is a resolution of singularities where $Y$ has either quasi--$F$--rational or klt and Cohen--Macaulay singularities, then $\omega_{\pi, \GR} = \omega_Y$. In particular, we can relax the smoothness assumption in \autoref{intro_main_thm} by one of these properties (assuming the existence of a resolution). 
\end{rem_blank}

Combining \autoref{intro_weak_GR_sheaf} and \autoref{intro_prop_euler_char}, we immediately obtain a proof of \autoref{intro_main_thm_ss}. Let us finish this introduction by explaining ideas around the proof of the second part of \autoref{intro_weak_GR_sheaf} (the first part is an explicit technical computation). There are three main observations to make: 

\begin{enumerate}
	\item $\chi_{ss}(X, \omega_X) = \chi(X, W\omega_X/p)$ and $\chi(A, (a_*W\omega_X)/p) = \chi_{ss}(A, \omega_{a, \GR})$;
	\item $\chi(X, W\omega_X/p) = \chi(X, W\omega_{X, \bQ})$ and $\chi(A, a_*W\omega_{X, \bQ}) = \chi(A, (a_*W\omega_X)/p)$;
	\item $\chi(X, W\omega_{X, \bQ}) = \chi(A, a_*W\omega_{X, \bQ})$.
\end{enumerate}

\noindent Combining these three observations give us that \[ \chi_{ss}(X, \omega_X) = \chi_{ss}(A, \omega_{a, \GR}),  \] so we are left to explain why these three statements hold. 

The first observation essentially follows from the definition of $\chi_{ss}$, as well as general properties of $W\omega_X$ and the fact that $\lim_e F^e_*\omega_{a, \GR} \cong (a_*W\omega_X)/p$, while the third observation is an immediate application of our Witt vector version of Grauert--Riemenschneider vanishing. 

Finally, the second observation is an incarnation of the fact that if $M^{\bullet}$ is a bounded complex of finitely generated $R$-modules with $(R, k)$ a regular local ring and $K \coloneqq \Frac(R)$, then \[ \chi_k(M^{\bullet} \otimes_R^L k) = \chi_K(M^{\bullet} \otimes_R^L K). \] A typical application, well--known to topologists, is that if $X$ is a (nice enough) topological space, then the Euler characteristic of $X$ with coefficients in a given field does not depend on this field (although the individual singular cohomology groups do!). In our case, we apply this to $R = W(k)$, and to $M^{\bullet} = \RGamma(X, W\omega_X)$ and $\RGamma(A, a_*W\omega_X)$. Of course, we need to prove beforehand that these cohomology groups are finitely generated. This is all carried in \autoref{section_weak_GR_sheaf}.

%
%
%
%
%

\subsection{Notations}

\begin{itemize}
	\item We fix once and for all a prime number $p > 0$. All rings and schemes in this paper are defined over $\bF_p$.
	
	\item The symbol $k$ always refers to a perfect field.
	
	\item A variety is an integral scheme $X$ which is separated and of finite type over $k$.
	
	\item The symbol $F$ always denotes the absolute Frobenius of a $\bF_p$-scheme. An $\bF_p$-scheme $X$ is $F$-finite if its absolute Frobenius morphism $F \colon X \to X$ is finite.

	\item Given a complex $C^{\bullet}$ with values in some abelian category and $i \in \bZ$, we denote by $\cH^i(C^{\bullet})$ its $i$-th cohomology object. When the context is clear, we may simply write $C$ instead of $C^{\bullet}$.
	
	\item Given a Noetherian ring $R$, we let $D_{\mathrm{fg}}(R)$ denote the derived category of $R$--modules with finitely generated cohomology modules.
	
	\item The symbol $A$ always denotes an abelian variety over an algebraically closed field, $\hat{A}$ denotes its dual, and $\cP$ denotes the normalized Poincaré bundle on $A \times \hat{A}$. Furthermore, given a closed point $\alpha \in \bighat{A}$, we write $\cP_\alpha \coloneqq \cP|_{A \times \alpha} \in \Pic^0(A)$.
	
	\item Given an $\bF_p$-scheme $X$ and $n \geq 1$, we write $W_nX$ for the scheme given by the ringed space $(X, W_n\cO_X)$. Given a morphism $f \colon X \to Y$, we still denote by $f$ the naturally induced morphism from $W_nX \to W_nY$.
	
	\item Given a separated morphism $f \colon X \to Y$ of finite type, we denote by $f^!$ the functor defined in \stacksproj{0A9Y}.
	
	\item Given a sheaf of abelian groups $\cF$ and an integer $m \geq 1$, the subsheaf $\cF[m] \inc \cF$ denotes the sheaf of $m$-torsion elements.
\end{itemize}

\subsection{Acknowledgments}
I would like to thank Fabio Bernasconi, Adrian Langer and Zsolt Patakfalvi for insightful discussions about the content of this article.

Financial support was provided by grant $\#$200020B/192035 from the Swiss National Science Foundation (FNS/SNF), and by grant $\#$804334 from the European Research Council (ERC).

\section{Preliminaries}

Here is a general lemma about Witt schemes that we will use throughout, without further mention.

\begin{lem}
	Let $f \colon X \to Y$ be a morphism of Noetherian, $F$--finite $\bF_p$--schemes.
	\begin{itemize}
		\item For all $n \geq 1$, the scheme $W_nX$ is also Noetherian, and the induced Frobenius $F \colon W_nX \to W_nX$ is finite.
		\item If $f$ is of finite type (resp. separated, proper), then so is the induced map $W_nX \to W_nY$.
	\end{itemize}
\end{lem}
\begin{proof}
	This is all contained in \cite[Appendix A.1]{Langer_Zink_De_Rham_Witt_cohomology_for_a_proper_and_smooth_morphism}.
\end{proof}

\subsection{Witt--Cartier modules and crystals}

Throughout, $X$ denotes a Noetherian $F$--finite $\bF_p$-scheme. 

\begin{defn}
	\begin{itemize}
		\item A \emph{$W_n$--Cartier module} is a pair ($\cM$, $\theta$) where $\cM$ is a coherent $W_n\cO_X$-module and $\theta \colon F_*\cM \to \cM$ is a morphism. We call $\theta$ the \emph{structural morphism} of $\cM$. 
		
		\item A morphism of $W_n$--Cartier modules $h \colon (\cM_1, \theta_1) \to (\cM_2, \theta_2)$ is a morphism of underlying $W_n\cO_X$-modules $h \colon \cM_1 \to \cM_2$ such $h \circ \theta_1 = \theta_2 \circ F_*h$. The category of $W_n$--Cartier modules is denoted $\Coh_{W_nX}^C$. When $n = 1$, we will call them Cartier modules, and denote their category by $\Coh_X^C$.
	\end{itemize}	
\end{defn}

\begin{rem}
	In other references (e.g. \cite{Baudin_Duality_between_Witt_Cartier_crystals_and_perverse_sheaves}), the objects we defined above are called coherent $W_n$--Cartier modules (the definition in \emph{loc. cit.} does not require (quasi)--coherence). Since all sheaves we will consider in this paper are coherent, this should not cause any trouble.
\end{rem}

Note that one can take the pushforward of a $W_n$--Cartier module: if $f \colon X \to Y$ is a proper morphism of Noetherian, $F$--finite $\bF_p$--schemes and $(\cM, \theta)$ is a $W_n$--Cartier module on $X$, then $(f_*\cM, f_*\theta)$ defines a $W_n$--Cartier module on $Y$ (the properness assumption is only here to ensure that $f_*\cM$ is coherent).

\begin{example}
	Arguably the most important example is the canonical $W_n$--dualizing sheaf $W_n\omega_X$, together with the Cartier operator, see \autoref{section_witt_dc}.
\end{example}

\begin{notation}\label{rem:def_Cartier_mod}
	\begin{enumerate}
		\item\label{itm:abuse_iterates_structure_map} Let $(\cM, \theta)$ be a $W_n$-Cartier module. We will abuse notations as follows: \[ \theta^e \coloneqq \theta \circ F_*\theta \circ \dots \circ F^{e-1}_*\theta \colon F^{e}_*\cM \to \cM. \]
		\item\label{itm:notation H^0_ss} Given a proper scheme $X$ over an algebraically closed field $k$ and $(\cM, \theta) \in \Coh_{W_nX}^{C}$, we set \[ H^i_{ss, \theta}(X, \cM) \coloneqq \bigcap_{e > 0} \im\left(H^i(X, F^{e}_*\cM) \xto{\theta^e} H^i(X, \cM)\right) \inc H^i(X, \cM). \] We also write $h^i_{ss, \theta}(X, \cM) \coloneqq \length H^i_{ss, \theta}(X, \cM)$ 
		and  \[ \chi_{ss, \theta}(\cM) \coloneqq \sum_i (-1)^ih^i_{ss, \theta}(X, \cM). \] Whenever the context is clear, we shall remove the $\theta$ from this notation.
	\end{enumerate}
\end{notation}

\begin{defn}\label{def:Cartier_crystal}
	A $W_n$--Cartier module is said to be \emph{nilpotent} if its structural morphism is nilpotent. These objects form a Serre subcategory of $\Coh_{W_nX}^{C}$ (meaning that being nilpotent is stable under subobjects, quotients and extensions). The corresponding quotient category (see \stacksproj{02MS}) is denoted by $\Crys_{W_nX}^{C}$, and the objects of this quotient category are called \emph{$W_n$--Cartier crystals}. 
\end{defn}

Intuitively, working in the category of crystals allows us to forget about nilpotent phenomena. For example, a $W_n$--Cartier module is the zero object as a $W_n$-Cartier crystal if and only if it is nilpotent. More generally, a morphism $f \colon \cM \to \cN$ of $W_n$-Cartier modules is an isomorphism at the level of crystals if and only if both $\ker(f)$ and $\coker(f)$ are nilpotent. Such a morphism is called a \emph{nil-isomorphism}.

\begin{notation}\label{notation_crystals}
	\begin{itemize}
		\item Morphisms in $\Crys_{W_nX}^C$ will be denoted by the symbol $\dashrightarrow$. Those which come from a morphism in $\Coh_{W_nX}^C$ will be denoted by a full arrow $\to$.
		\item If $\cM_1$ and $\cM_2$ are two $W_n$--Cartier modules which are isomorphic as crystals, then we shall write $\cM_1 \sim_C \cM_2$.
	\end{itemize}
\end{notation}

\begin{defn}
	Let $(\cM, \theta)$ be a $W_n$-Cartier module. Then the \emph{perfection} of $(\cM, \theta)$ is $\cM^{\perf} \coloneqq \lim_e F^e_*\cM$.
\end{defn}

\begin{lem}\label{Gabber_finiteness}
	Let $\cM$ be a $W_n$--Cartier module on $X$.
	\begin{enumerate}
		\item\label{itm:Gabber_fin} The inverse system $\{F^{e}_*\cM\}_{e \geq 1}$ satisfies the Mittag--Leffler condition. In particular, $\cM^{\perf} = \Rlim F^{e}_*\cM$ and perfection is exact on $W_n$--Cartier modules.
		\item We have that $\cM \sim_C 0$ if and only if $\cM^{\perf} = 0$.
		\item\label{itm:nil_iso_iff_perf_iso} For any morphism $f \colon \cM \to \cN$ of $W_n$--Cartier modules, $f$ is a nil--isomorphism if and only if $f^{\perf}$ is an isomorphism.
	\end{enumerate}
\end{lem}
\begin{proof}
	This is \cite[Lemma 4.4.2]{Baudin_Duality_between_Witt_Cartier_crystals_and_perverse_sheaves} (see also \cite[Lemma 13.1]{Gabber_notes_on_some_t_structures} for the first statement).
\end{proof}

\begin{rem}\label{rem_ss_cohomology}
	Given a Cartier module $\cM$ on a proper scheme $X$ over an algebraically closed field, it follows from \autoref{Gabber_finiteness}.\autoref{itm:Gabber_fin} that $H^i_{ss}(X, \cM) = H^i(X, \cM^{\perf})$. In particular, the groups $H^i(X, \cM^{\perf})$ are finite--dimensional.
\end{rem}

\subsection{Witt dualizing complexes}\label{section_witt_dc}

Throughout, $X$ denotes a fixed variety with structural morphism $g \colon X \to \Spec k$. Here is our main object of study:

\begin{defn}
	The \emph{canonical $W_n$-dualizing complex} is by definition $W_n\omega_X^{\bullet} \coloneqq g^!\cO_{W_n(k)}$. By \stacksproj{0AA3}, this is a dualizing complex on $W_nX$.
\end{defn}
\begin{rem}
	When $X$ is smooth, this complex agrees with the last piece of the de Rham--Witt complex from \cite{Illusie_Complexe_de_de_Rham_Witt_et_cohomologie_cristalline} (up to a shift), see \cite[Section I]{Ekedahl_Duality_Hodge_Witt}.
\end{rem}

In \cite[Section 2.2]{Baudin_Witt_GR_vanishing_and_applications} (see also \cite[Appendix]{Quasi_F_splittings_III}), we explain how to construct the Cartier operators $C \colon F_*W_n\omega_X^{\bullet} \to W_n\omega_X^{\bullet}$ and restriction maps $R \colon W_{n + 1}\omega_X^{\bullet} \to W_n\omega_X^{\bullet}$, satisfying $R \circ C = C \circ F_*R$. In particular, each sheaf $\cH^i(W_n\omega_X^{\bullet})$ is a $W_n$--Cartier module, and each map $\cH^i(R) \colon \cH^i(W_{n + 1}\omega_X^{\bullet}) \to \cH^i(W_n\omega_X^{\bullet})$ is a morphism of $W_{n + 1}$--Cartier modules. 

\begin{notation}
	We set \[ (W_n\omega_X^{\bullet})^{\perf} \coloneqq \Rlim_e F^e_*W_n\omega_X^{\bullet}, \] where the transition maps are given by the Cartier operators $C$. Note that by exactness of perfection (see \autoref{Gabber_finiteness}.\autoref{itm:Gabber_fin}), this does not cause any clash of notation with the perfection functor on $W_n$--Cartier modules.
\end{notation}

Since the restriction maps commute with the Cartier operators, there are induced morphisms $R \colon (W_{n + 1}\omega_X^{\bullet})^{\perf} \to (W_n\omega_X^{\bullet})^{\perf}$.

\begin{defn}
	We define the \emph{canonical Witt--dualizing complex} to be \[ W\omega_X^{\bullet} \coloneqq \Rlim \: W_n\omega_X^{\bullet}, \] where the transition maps are the restrictions $R \colon W_{n + 1}\omega_X^{\bullet} \to W_n\omega_X^{\bullet}$.
\end{defn}

An important (any maybe surprising) feature of canonical Witt dualizing complexes is the following:
\begin{lem}\label{Witt_omega_coperfect}
	The induced morphism $C \colon F_*W\omega_X^{\bullet} \to W\omega_X^{\bullet}$ is an isomorphism. In particular, there is a natural isomorphism \[ W\omega_X^{\bullet} \to \Rlim_n \: (W_n\omega_X^{\bullet})^{\perf}. \]
\end{lem} 
\begin{proof}
	This is \cite[Lemma 2.2.5]{Baudin_Witt_GR_vanishing_and_applications}.
\end{proof}

\begin{lem}\label{cone_of_p_Witt_omega}
	The following is an exact triangle
	\[ \begin{tikzcd}
		W\omega_X^{\bullet} \arrow[r, "p^m"] & W\omega_X^{\bullet} \arrow[r] & (W_m\omega_X^{\bullet})^{\perf} \arrow[r, "+1"] & {}
	\end{tikzcd} \] where the second map is the composition $W\omega_X^{\bullet} \xto{\cong} \Rlim_n \: (W_n\omega_X^{\bullet})^{\perf} \to (W_m\omega_X^{\bullet})^{\perf}$.
\end{lem}

\begin{proof}
	This is \cite[Lemma 2.2.7]{Baudin_Witt_GR_vanishing_and_applications}.
\end{proof}

\begin{defn}
	The \emph{canonical $W_n$--dualizing sheaf} is $W_n\omega_X \coloneqq \cH^{-\dim(X)}(W_n\omega_X^{\bullet})$. We also set the \emph{canonical Witt--dualizing sheaf} to be $W\omega_X \coloneqq \lim_n W_n\omega_X$.
\end{defn}

\begin{rem}\label{coincides_with_what_we_know}
	If $X$ is smooth, then $W\omega_X$ agrees with the last piece of the (non--truncated) de Rham--Witt complex of \cite{Illusie_Complexe_de_de_Rham_Witt_et_cohomologie_cristalline} up to a shift by \cite[Remark 2.2.12]{Baudin_Witt_GR_vanishing_and_applications}.
\end{rem}

\section{A weak Grauert-Riemenschneider sheaf}\label{section_weak_GR_sheaf}

Throughout, we fix a proper and generically finite morphism $\pi \colon X \to Y$ of varieties with $X$ smooth. We want to find a canonically defined sub--Cartier crystal $\omega_{\pi, \GR} \inc \pi_*\omega_X$ satisfying $\chi_{ss}(X, \omega_X) = \chi_{ss}(Y, \omega_{\pi, \GR})$ (when $X$ and $Y$ are proper). The essential ingredient will be our recent $\bQ_p$--Grauert--Riemenschneider vanishing theorem (in short, $\bQ_p$--GR vanishing):

\begin{sthm}[{\cite{Baudin_Witt_GR_vanishing_and_applications}}]\label{Witt_GR_vanishing}
	In our setup above, there exists $e > 0$ such that for all $i > 0$, \[ p^e \cdot R^i\pi_*W\omega_X = 0. \]
\end{sthm}

\noindent Let us now move to the construction and properties of $\omega_{\pi, \GR}$.

\begin{slem}\label{reinterpretation_of_inverse_system}
	There is a natural isomorphism \[ \pi_*W\omega_X \to \Rlim \pi_*W_n\omega_X^{\perf}. \] 
\end{slem}
\begin{proof}
	By \autoref{Witt_omega_coperfect}, there is a natural isomorphism $W\omega_X \to \lim W_n\omega_X^{\perf}$. Applying $\pi_*$ gives that $\pi_*W\omega_X \to \lim \pi_*W_n\omega_X^{\perf}$ is an isomorphism. To conclude the proof, it is then enough to show that the inverse system $\{\pi_*W_n\omega_X^{\perf}\}_{n \geq 1}$ satisfies the Mittag--Leffler condition. This follows from \cite[Remark 2.2.11]{Baudin_Witt_GR_vanishing_and_applications}.
\end{proof}

\begin{snotation}
	Let $\cF^{\bullet}$ be a complex of abelian sheaves. In order to lighten notations, we will write \[ \cF^{\bullet}/p^n \coloneqq \cF^{\bullet} \otimes_{\bZ}^L \bZ/p^n\bZ. \]
\end{snotation}

\begin{srem}\label{rem_derived_p_completion_first}
	Note that $\cF^{\bullet}/p^n$ is the cone of $\cF^{\bullet} \xto{p^n} \cF^{\bullet}$. Since all our functors are additive (hence respect multiplication by $p^n$), they will all preserve derived tensor product by $\bZ/p^n\bZ$.
\end{srem}

\begin{sprop}\label{cohomology_fg_mod_p}
	There exists a canonically defined sub--Cartier crystal $\cM$ of $\pi_*\omega_X$ such that \[ (\pi_*W\omega_X)/p \cong \cM^{\perf}. \]
\end{sprop}

\begin{proof}
	For any $n \geq 1$, let $\cM_n \subseteq \pi_*\omega_X$ be the sub--$W_n$--Cartier module defined by the image of \[ \pi_*W_n\omega_X \to \pi_*\omega_X. \] 
	Since it is a $p$--torsion $W_n$--Cartier module, we know by \cite[Lemma 4.3.8]{Baudin_Duality_between_Witt_Cartier_crystals_and_perverse_sheaves} that there exists a Cartier module $\cM_n'$, unique in the category of Cartier crystals that is equipped with a nil--isomorphism $\cM_n \to \cM_n'$ of $W_n$--Cartier modules. Again by \emph{loc.cit.}, there exists a injection of Cartier crystals $\cM_n' \dra \pi_*\omega_X$, and since we have inclusions \[ \dots \inc \cM_{n + 1} \inc \cM_n \inc \dots \inc \pi_*\omega_X, \] we obtain a descending chain of Cartier crystals \[ \dots \dra \cM_{n + 1}' \dra \cM_n' \dra \dots \dra \pi_*\omega_X. \]	
	
	Since Cartier crystals are Artinian objects (see e.g. \cite[Main theorem]{Blickle_Bockle_Cartier_modules_finiteness_results}), this descending chain must eventually stabilize. Set $\cM \coloneqq \cM_n'$ for $n \gg 0$ (it is only defined as a Cartier crystal!). By construction, we obtain that for $n \gg 0$, the inclusions $\cM_{n + 1} \inc \cM_n$ are nil-isomorphisms, so they induce isomorphisms \[ \cM_{n + 1}^{\perf} \cong \cM_n^{\perf} \] by \autoref{Gabber_finiteness}.\autoref{itm:nil_iso_iff_perf_iso}, and they agree with $\cM^{\perf}$. To finish the proof, we need to show that \[ (\pi_*W\omega_X)/p \cong \cM^{\perf}.\] We have that \[ (\pi_*W\omega_X)/p \expl{\cong}{\autoref{reinterpretation_of_inverse_system}} \left(\Rlim \pi_*W_n\omega_X^{\perf}\right)/p \expl{\cong}{\autoref{rem_derived_p_completion_first}} \Rlim\left((\pi_*W_n\omega_X)^{\perf}/p\right) \expl{\cong}{\autoref{rem_derived_p_completion_first} and \autoref{Gabber_finiteness}} \Rlim \left( (\pi_*W_n\omega_X)/p\right)^{\perf}. \] 
	First of all, we claim that \[ \Rlim \left(\left( (\pi_*W_n\omega_X)/p\right)^{\perf}\right) \cong \Rlim \cH^0\left((\pi_*W_n\omega_X)/p\right)^{\perf}. \] To see this, all we have to do is showing that $\Rlim \left(\cH^{-1}(\pi_*W_n\omega_X/p)^{\perf}\right) = 0$. The fact that $R^1\lim$ vanishes follows from \cite[Remark 2.2.11]{Baudin_Witt_GR_vanishing_and_applications}. Finally, the $\lim$ also vanishes, since \[ \lim \left(\cH^{-1}(\pi_*W_n\omega_X/p)^{\perf}\right) = \cH^{-1}\left(\Rlim \left( (\pi_*W_n\omega_X)/p\right)^{\perf}\right) \expl{\cong}{see above} \cH^{-1}\left((\pi_*W\omega_X)/p\right) = 0, \] where the last vanishing holds since $\pi_*W\omega_X$ is $p$--torsion free (see \cite[Lemma 2.2.9]{Baudin_Witt_GR_vanishing_and_applications}). Hence, our claim is proven.


	Note that by construction, we have a commutative diagram \[ \begin{tikzcd}
		\Rlim \cH^0((\pi_*W_n\omega_X^{\perf})/p) \arrow[rr]                 &  & \Rlim \cM_n^{\perf} = \cM^{\perf} \arrow[d, "\inc"] \\
		(\pi_*W\omega_X)/p \arrow[u, "\cong"] \arrow[rr, hook] &  & \pi_*\omega_X^{\perf},                            
	\end{tikzcd} \] showing that the top arrow is injective. Since it is an inverse limit of surjective maps and there is no $R^1\lim$ in play by \cite[Remark 2.11]{Baudin_Witt_GR_vanishing_and_applications}, we deduce that the top arrow is in fact an isomorphism. All in all, we have shown that $(\pi_*W\omega_X)/p \cong \cM^{\perf}$.
\end{proof}

\begin{sdefn}
	The Cartier crystal $\cM$ in \autoref{cohomology_fg_mod_p} will be called the \emph{weak Grauert--Riemenschneider sheaf} (in short, weak GR sheaf) with respect to $\pi$. It will be denoted $\omega_{\pi, \GR}$.
\end{sdefn}

\begin{srem}\label{rem_we_may_assume_sub_Cartier_mod}
	By its very construction, $\omega_{\pi, \GR}$ is only defined as a sub--Cartier \emph{crystal} of $\pi_*\omega_X$, not as a sub--Cartier module. In practice, this might be a bit cumbersome, so let us explain how to see it as a sub--Cartier \emph{module} of $\pi_*\omega_X$. 
	
	Since we have an injective morphism of Cartier crystals $\omega_{\pi, \GR} \dra \pi_*\omega_X$, we know by \cite[Lemma 3.4.14]{Baudin_Duality_between_perverse_sheaves_and_Cartier_crystals} that we can fit it in a commutative diagram \[ \begin{tikzcd}
		F^e_*\omega_{\pi, \GR} \arrow[d] \arrow[rd, "g"] &               \\
		\omega_{\pi, \GR} \arrow[r, dashed]         & \pi_*\omega_X,
	\end{tikzcd} \] where the left vertical map is the structural morphism of $\omega_{\pi, \GR}$ (hence an isomorphism of crystals), and $g$ is a morphism of Cartier \emph{modules}. Since $g$ is also injective as a morphism of crystals, $\ker(g)$ is nilpotent. Hence, if we set $\ttilde{\omega_{a, \GR}} \coloneqq \im(g) \inc \pi_*\omega_X$, then $\omega_{\pi, \GR} \sim_C \ttilde{\omega_{\pi, \GR}} \inc \pi_*\omega_X$. \\
\end{srem}

Our goal is now to show that $\chi_{ss}(X, \omega_X) = \chi_{ss}(Y, \omega_{\pi, \GR})$ when $X$ and $Y$ are proper. An important intermediate step is to show that each $W(k)$--module $H^i(Y, \pi_*W\omega_X)$ is finitely generated, and the key notion to prove this is derived $p$--completeness. The reason is that a derived $p$--complete complex of $W(k)$--modules which is finitely generated modulo $p$ is in fact finitely generated (\stacksproj{0CQF}).

\begin{sdefn}[{\stacksproj{091Z}}]
	We say that a complex of abelian sheaves $\cF^{\bullet}$ is \emph{derived $p$-complete} if the natural map \[ \cF^{\bullet} \to \Rlim_n \: \cF^{\bullet}/p^n \] is an isomorphism.
\end{sdefn}

\begin{srem}\label{rem_derived_p_completion_second}
	Since derived pushforwards preserve $\Rlim$ and $- \otimes^L_{\bZ} \bZ/p^n\bZ$ (see \autoref{rem_derived_p_completion_first}), they also preserve derived $p$--complete objects.
\end{srem}

\begin{slem}\label{Witt_omega_derived_p_complete}
	The sheaf $W\omega_X$ is derived $p$--complete.
\end{slem}
\begin{proof}
	For any variety $X$, it follows from \autoref{cone_of_p_Witt_omega} and \autoref{Witt_omega_coperfect} that $W\omega_X^{\bullet}$ is derived $p$--complete. Note that $W\omega_X^{\bullet}$ is simply the sheaf $W\omega_X$ shifted in some degree in our case (this holds for each $W_n\omega_X^{\bullet}$ and $W_n\omega_X$ by the proof of \cite[Lemma 2.2.15]{Baudin_Witt_GR_vanishing_and_applications}, so apply Lemma 2.2.10 in \emph{loc.cit.}). The proof is then complete.
\end{proof}

\begin{slem}\label{pushforward_p_complete}
	The sheaf $\pi_*W\omega_X$ is derived $p$-complete. In particular, $\RGamma(Y, \pi_*W\omega_X)$ is also derived $p$-complete.
\end{slem}

\begin{proof}
	By \autoref{Witt_omega_derived_p_complete}, we know that $W\omega_X$ is derived $p$--complete. In particular, $R\pi_*W\omega_X$ is also derived $p$-complete by \autoref{rem_derived_p_completion_second}, so the canonical map \[ \begin{tikzcd}
		\pi_*W\omega_X \arrow[rr, "\theta"] &  & \cH^0\left(\Rlim R\pi_*W\omega_X/p^n\right)
	\end{tikzcd} \] is an isomorphism. For all $n \geq 1$, there is an exact triangle \[  \begin{tikzcd}
		(\pi_*W\omega_X)/p^n \arrow[rr] &  & (R\pi_*W\omega_X)/p^n \arrow[rr] 	&  & (\tau_{\geq 1}R\pi_*W\omega_X)/p^n \arrow[rr, "+1"] &  & {}
	\end{tikzcd} \] so applying $\Rlim$ gives an exact triangle \\
	\begin{equation}\label{triangle_p_compl}
		\Rlim \left(\left(\pi_*W\omega_X\right)/p^n\right) \xto{\psi} \Rlim \left(\left(R\pi_*W\omega_X\right)/p^n\right) \to \Rlim \left(\left(\tau_{\geq 1}R\pi_*W\omega_X\right)/p^n\right) \xto{+1} 
	\end{equation}
	Since $\pi_*W\omega_X$ is $p$--torsion free (\cite[Lemma 2.2.9]{Baudin_Witt_GR_vanishing_and_applications}), the complex $(\pi_*W\omega_X)/p^n$ is actually a sheaf in degree zero (and it agrees with the usual non--derived quotient). In particular, each map of sheaves $(\pi_*W\omega_X)/p^{n + 1} \to (\pi_*W\omega_X)/p^n$ is surjective, so $\lim \: (\pi_*W\omega_X)/p^n = \Rlim \: (\pi_*W\omega_X)/p^n$.
	
	By construction, we have a commutative diagram \[ \begin{tikzcd}
		\pi_*W\omega_X \arrow[rr, "\theta"] \arrow[rd] &                                        & \cH^0\left(\Rlim \: (R\pi_*W\omega_X)/p^n\right), \\
		& \lim \: (\pi_*W\omega_X)/p^n \arrow[ru, "\cH^0(\psi)"'] &                                    
	\end{tikzcd} \] so in order to show that $\pi_*W\omega_X$ is derived $p$-complete, it suffices to show that $\cH^0(\psi)$ is an isomorphism (see the previous paragraph).
	Note that there is a commutative diagram \[ \begin{tikzcd}
		{(R^1\pi_*W\omega_X)[p^{n + 1}]} \arrow[rr, "\cdot p"]                                              &  & {(R^1\pi_*W\omega_X)[p^n]}                                                         \\
		\cH^0\left(\left(\tau_{\geq 1}R\pi_*W\omega_X\right)/p^{n + 1}\right) \arrow[rr] \arrow[u, "\cong"] &  & \cH^0\left(\left(\tau_{\geq 1}R\pi_*W\omega_X\right)/p^n\right) \arrow[u, "\cong"]
	\end{tikzcd} \] (this is not specific to $R\pi_*W\omega_X$, this fact holds for any complex of abelian sheaves). Since there exists $e > 0$ such that $p^e \cdot R^1\pi_*W\omega_X = 0$ by $\bQ_p$--GR vanishing (\autoref{Witt_GR_vanishing}), we deduce that \[ \cH^0\left(\Rlim\left(\left(\tau_{\geq 1}R\pi_*W\omega_X\right)/p^n\right)\right) = \lim\cH^0\left(\left(\tau_{\geq 1}R\pi_*W\omega_X\right)/p^n\right) = \lim \: R^1\pi_*W\omega_X[p^n] = 0. \] By the triangle \autoref{triangle_p_compl}, we deduce that $\cH^0(\psi)$ is an isomorphism.
	
	Thus, we have proven that $\pi_*W\omega_X$ is derived $p$--complete. The fact that $\RGamma(Y, \pi_*W\omega_X)$ is also derived $p$--complete follows from \autoref{rem_derived_p_completion_second}.
\end{proof}

\begin{srem}
	It seems likely to the author that $\bQ_p$--GR vanishing was not necessary in the proof above. It could be that in general, the $p^{\infty}$--torsion of higher pushforwards of $W\omega_X$ under any proper map is annihilated by some fixed $p$--power. Nevertheless, we will need the full $\bQ_p$--GR vanishing later.
\end{srem}

\begin{scor}\label{finiteness_of_cohomology}
	Assume that $X$ and $Y$ are proper.  Then for all $i \geq 0$, the $W(k)$--module $H^i(Y, \pi_*W\omega_X)$ is finitely generated, and there is an isomorphism \[ H^i\left(Y, (\pi_*W\omega_X)/p\right) \cong H^i_{ss}(Y, \omega_{\pi, \GR}). \]
\end{scor}
\begin{proof}
	We have \[ H^i\left(Y, (\pi_*W\omega_X)/p\right) \expl{\cong}{by definition} H^i(Y, \omega_{\pi, \GR}^{\perf}) \expl{\cong}{\autoref{rem_ss_cohomology}} H^i_{ss}(Y, \omega_{\pi, \GR}), \] and we know that this latter $k$--vector space is finitely generated by properness. In particular, \[ \RGamma(Y, \pi_*W\omega_X)/p \expl{\cong}{\autoref{rem_derived_p_completion_first}} \RGamma(Y, (\pi_*W\omega_X)/p) \in D_{\mathrm{fg}}(k), \] so we conclude by \stacksproj{0CQF} that \[ \Rlim \RGamma(Y, \pi_*W\omega_X)/p^n \in D_{\mathrm{fg}}(W(k)). \] By \autoref{pushforward_p_complete}, this complex above is exactly $\RGamma(Y, \pi_*W\omega_X)$.
\end{proof}

\begin{notation_sec}
	Let $L$ be a field, and let $V^{\bullet} \in D^b_{\mathrm{fg}}(L)$. We set \[ \chi_L(V^{\bullet}) \coloneqq \sum_{i \in \bZ} (-1)^i\dim_L(\cH^i(V^{\bullet})). \]
\end{notation_sec}

Now comes the trick that allows us to compare the semistable Euler characteristics of $\omega_X$ and $\omega_{\pi, \GR}$.
\begin{slem}\label{Linus_paper_trick}
	Let $M^{\bullet} \in D^b_{\mathrm{fg}}(W(k))$. Then \[ \chi_K(M^{\bullet} \otimes^L_{W(k)} K) = \chi_k(M^{\bullet}/p). \]
\end{slem}
\begin{proof}
	This is standard, see e.g. the proof of \cite[Lemma 2.20]{Baudin_Patakfalvi_Rosler_Zdanowicz_On_Gorenstein_Q_p_rational_threefolds_and_fourfolds}.
\end{proof}

\begin{scor}\label{weak_GR_sheaf}
	Assume that $X$ and $Y$ are proper. Then \[ \chi_{ss}(Y, \omega_{\pi, \GR}) = \chi_{ss}(X, \omega_X). \]
\end{scor}
\begin{proof}
	We have that \begin{align*}
		\chi_{ss}(Y, \omega_{\pi, \GR}) & \expl{=}{\autoref{rem_ss_cohomology}} \chi(Y, (\pi_*W\omega_X)/p) = \chi_k\left(\RGamma(Y, (\pi_*W\omega_X)/p)\right) = \chi_k\left(\RGamma(Y, \pi_*W\omega_X)/p\right) \\
		& \expl{=}{\autoref{finiteness_of_cohomology} and \autoref{Linus_paper_trick}} \chi_K\left(\RGamma(Y, \pi_*W\omega_X) \otimes^L_{W(k)} K\right) \expl{=}{\autoref{Witt_GR_vanishing}} \chi_K\left(\RGamma(X, W\omega_X) \otimes^L_{W(k)} K\right) \\
		& \expl{=}{\autoref{Linus_paper_trick}} \chi_k\left(\RGamma(X, W\omega_X)/p\right)
		\expl{=}{\autoref{cone_of_p_Witt_omega}} \chi_k\left(\RGamma(X, \omega_X^{\perf})\right) \expl{=}{\autoref{rem_ss_cohomology}} \chi_{ss}(X, \omega_X). \qedhere
	\end{align*} 
\end{proof}


In fact, there is the following generalization which follows from our techniques. We will use it in the proof of \autoref{euler_char_zero_implies_alb_fibered_by_ab_vars}:

\begin{scor}\label{weak_GR_sheaf_on_steroids}
	Assume that $X$ and $Y$ are proper. Let $L \in \Pic(Y)$ be a line bundle satisfying the following properties: for all $i \geq 0$, the vector spaces \[ H^i(X, \omega_X^{\perf} \otimes \pi^*L) \] and \[ H^i(Y, \omega_{\pi, \GR}^{\perf} \otimes L) \] are finitely generated. Then \[ \chi(Y, \omega_{\pi, \GR}^{\perf} \otimes L) = \chi(X, \omega_X^{\perf} \otimes \pi^*L). \]
\end{scor}
\begin{proof}
	Let $\underline{L}$ denote the Teichmüller lift of $L$ (it is then a line bundle over the ringed space $(Y, W\cO_Y)$). Note that both $W\omega_X \otimes_{W\cO_X} \pi^*\underline{L}$ are $\pi_*W\omega_X \otimes_{W\cO_Y} \underline{L}$ are derived $p$-complete (being derived $p$-complete is a local statement, so this reduces to \autoref{pushforward_p_complete}). In particular $\RGamma(X, W\omega_X \otimes_{W\cO_X} \pi^*\underline{L})$ and $\RGamma(Y, \pi_*W\omega_Y \otimes_{W\cO_Y} \underline{L})$ are also derived $p$--complete.
	
	Note that twisting the triangle of \autoref{cone_of_p_Witt_omega} by $\pi^*\underline{L}$ shows that $(W\omega_X \otimes_{W\cO_X} \pi^*\underline{L})/p \cong \omega_X^{\perf} \otimes_{\cO_X} \pi^*L$, and similarly $(\pi_*W\omega_X \otimes_{W\cO_Y} \underline{L})/p \cong \omega_{\pi, \GR}^{\perf} \otimes_{\cO_Y} L$. Hence, the assumption and the same argument as in \autoref{finiteness_of_cohomology} imply that each $W(k)$--module $H^i(X, W\omega_X \otimes \pi^*\underline{L})$ is finitely generated, and similarly for $H^i(Y, \pi_*W\omega_X \otimes_{W\cO_Y} \underline{L})$.
	
	Finally, note that for all $i > 0$, the sheaves $R^i\pi_*(W\omega_X \otimes_{W\cO_X} \pi^*\underline{L}) \cong R^i\pi_*W\omega_X \otimes_{W\cO_Y} \underline{L} = 0$ are $p$--power torsion by \autoref{Witt_GR_vanishing} (we use the projection formula for the morphism of ringed spaces $\pi \colon (X, W\cO_X) \to (Y, W\cO_Y)$, see \stacksproj{01E8}). Thus, we can apply the exact same proof as in \autoref{weak_GR_sheaf}. 
\end{proof}

If $L$ is ample or anti-ample, there seems to be no chance that the hypotheses of \autoref{weak_GR_sheaf_on_steroids} hold in general. Nevertheless, it holds in the intermediate situation: 

\begin{slem}\label{cohom_perfection_twist_fg}
	Let $\cM$ be a Cartier module on a proper variety $X$. Then for all $L \in \Pic^{\tau}(X)$ and $i \geq 0$, the $k$--vector spaces $H^i(X, \cM^{\perf} \otimes L)$ are finitely generated.
\end{slem}
\begin{proof}
	Since line bundles in  $\Pic^{\tau}(X)$ form a bounded family, the quantity \[ \sup_{L \in \Pic^{\tau}(X)}\{ h^i(X, \cM \otimes L) \} \] is finite. Hence, the lemma is immediate from the projection formula, and the standard fact that if $\{V_n\}_{n \geq 1}$ is an inverse system of $k$--vector spaces such that $\dim(V_n)$ is bounded independently on $n$, then $\lim V_n$ is also finite--dimensional.
\end{proof}

\section{Euler characteristic of Cartier modules on abelian varieties}\label{section_Euler_char_V_modules}

Our goal is to show that on an ordinary abelian variety $A$ (see \autoref{def_ordinarity_AV}), any Cartier module has non-negative semistable Euler characteristic. Since Cartier modules are ``GV-sheaves up to nilpotence'' (see \cite{Hacon_Pat_GV_Characterization_Ordinary_AV, Baudin_Positive_characteristic_generic_vanishing_theory}), a first idea would be mimic the proof that GV--sheaves have non-negative Euler characteristic.

Unfortunately, working with semistable Euler characteristic (which we must if we want to use generic vanishing techniques) makes the situation subtler: it is not always true that Cartier modules on \emph{any} abelian variety have non-negative semistable Euler characteristic. For example, if $E$ is a supersingular elliptic curve, then $\chi_{ss}(E, \omega_E) = -1$.

The fundamental issue is that this semistable Euler characteristic is \emph{not} invariant under twists as above, see \cite[Remark 3.4.3]{Baudin_Positive_characteristic_generic_vanishing_theory}.

\subsection{$V$-modules and generic vanishing}

\begin{defn}[{\cite[Section 5]{Hacon_Pat_GV_Geom_Theta_Divs}}]\label{def:V-module}
	The \emph{relative Verschiebung} $V_{A/k} \colon \bighat{A}^{(p)} = \bighat{A^{(p)}} \to \bighat{A}$ is by definition the dual isogeny of the relative Frobenius $F_{A/k} \colon A \to A^{(p)}$, see \stacksproj{0CC9}.
	
	Since $k$ is perfect, we can canonically identify $\bighat{A}^{(p)}$ and $\bighat{A}$ as schemes. The induced morphism $V \colon \bighat{A} \to \bighat{A}$ is called the \emph{absolute Verschiebung}. By definition, the diagram
	\[ \begin{tikzcd}
		\bighat{A} \arrow[rr, "V"] \arrow[d] &  & \bighat{A} \arrow[d] \\
		\Spec k \arrow[rr, "F^{-1}"]         &  & \Spec k             
	\end{tikzcd} \] commutes.
	\begin{itemize}
		\item A \emph{$V$-module} on $\bighat{A}$ is a pair $(\cN, \theta)$ where $\cN$ is a coherent sheaf on $\bighat{A}$ with a morphism $\theta \colon \cN \to V^*\cN$. A morphism of $V$-modules is defined as for Cartier modules: it is a morphism of underlying sheaves commuting with the $V$-module structures. The category of $V^s$-modules on $\bighat{A}$ is denoted $\Coh_{\bighat{A}}^{V}$.
		\item A $V$-module $(\cN, \theta)$ is nilpotent if $\theta^N = 0$ for some $N \geq 0$ (we abuse notations as in \autoref{rem:def_Cartier_mod}.\autoref{itm:abuse_iterates_structure_map}). Nilpotent modules form a Serre subcategory, whose quotient is the category of \emph{$V$-crystals}, denoted $\Crys_{\bighat{A}}^{V}$. If $\cN_1$ and $\cN_2$ are $V$-modules which are isomorphic as crystals, we shall write $\cN_1 \sim_V \cN_2$. 
		\item A $V$-module $(\cN, \theta)$ is said to be \emph{injective} if its structural map $\theta$ is injective.
	\end{itemize}
\end{defn}
\begin{rem}\label{canonical_injective_V_module}
	\begin{itemize}
		\item For any $V$-module $\cN$, there exists a canonical injective $V$-module $\cN_{\injj}$ with a surjective map $\cN \surj \cN_{\injj}$, inducing an isomorphism of crystals. Explicitly, $\cN_{\injj}$ is given by the image of $\cN \to \colim V^{es, *}\cN$.
		
		Note that the induced morphism $\cN \to \cN_{\injj}$ is indeed an isomorphism of $V$-crystals, since it is surjective and by construction, $\ker(f)$ is a nilpotent $V$-module.
	\end{itemize}
\end{rem}

\begin{notation}\label{not_euler_char_V_mods}
	Let $\cN$ be a $V$-module. We define \[ \chi_{ss}(\cN) \coloneqq \sum_i (-1)^i \dim_k\left(\colim \Tor_i(V^{e, *}\cN, k(0))\right). \]
\end{notation}

\begin{rem}
	By resolving $\cN$ by a complex of free modules at $\cO_{\bighat{A}, 0}$, we see that the $k$--vector spaces $\Tor_i(V^{e, *}\cN, k(0))$ have bounded dimension, independently of $e$. In particular, their colimit is also finite--dimensional (this is a general result about systems of vector spaces, which we let the reader verify).
\end{rem}

\begin{defn}\label{def_ordinarity_AV}
	An abelian variety is said to be \emph{ordinary} if the natural morphism $H^1(A, \cO_A) \to H^1(A, F_*\cO_A)$ induced by the Frobenius is an isomorphism.
%
\end{defn}


\begin{prop}\label{equivalence_ordinary}
	An abelian variety $A$ is ordinary if and only if the absolute Verschiebung $V$ is étale.
\end{prop}
\begin{proof}
	This is classical, see e.g. \cite[Proposition 2.3.2]{Hacon_Pat_GV_Characterization_Ordinary_AV} and references therein.
\end{proof}

The idea is to rephrase our sought statement on Cartier modules by a statement on $V$--modules, which is much easier to handle. The key to do this is to use the (symmetric) \emph{Fourier--Mukai transform} on Cartier modules.

\begin{defn}[{\cite{Mukai_Fourier_Mukai_transform, Schnell_Fourier_Mukai_transform_made_easy}}]
	\begin{itemize}
		\item The \emph{Fourier-Mukai transforms} are the functors $R\bighat{S} \colon D_{\coh}(A) \to D_{\coh}(\bighat{A})$ and $RS \colon D_{\coh}(\bighat{A}) \to D_{\coh}(A)$ defined by 
		\[ R\hat{S}(\cM) \coloneqq Rp_{\hat{A}, *}(Lp_A^*\cM \otimes \cP)  \esp \mbox{ and } \esp RS(\cN) \coloneqq Rp_{A, *}(Lp_{\bighat{A}}^*\cN \otimes \cP), \] where $p_A \colon A \times \bighat{A} \to A$ and $p_{\bighat{A}} \colon A \times \bighat{A} \to \bighat{A}$ denote the projections.
		\item The \emph{symmetric Fourier-Mukai transforms} are the functors $\FM_A \coloneqq R\bighat{S} \circ D_A$ and $\FM_{\bighat{A}} \coloneqq RS \circ D_{\bighat{A}}$, where $D_A \coloneqq \cR\HHom(-, \omega_A[g])$ (and similarly for $D_{\bighat{A}}$).
	\end{itemize}
\end{defn}

Although the precise definition of the (symmetric) Fourier-Mukai transforms may look technical, we will never use it explicitly in this paper. We will rather use some of its well--known properties, that we recall now.

\begin{thm}\label{properties_Fourier_Mukai_transform}
	\begin{enumerate}
		\item\label{itm:equiv_cat} The functors $\FM_A$ and $\FM_{\bighat{A}}$ are inverses of each other, and hence equivalences of categories.
		\item\label{itm:behaviour_pushforwards_and_pullbacks} Let $f \colon A \to B$ be a morphism of abelian varieties, and let $\bighat{f} \colon \bighat{B} \to \bighat{A}$ denote the dual morphism. Then we have \[ \FM_B \: \circ \: Rf_* = L\bighat{f}^* \circ \FM_A. \]
		Here are two special cases:
		\begin{itemize}
			\item if $\cM$ is a coherent sheaf on $A$, then for all $i \geq 0$, \[ H^i(A, \cM)^{\vee} \cong \Tor_i(\FM_A(\cM), k(0)); \]
			\item if $\cM$ is a Cartier module, then each $\cH^i(\FM_A(\cM))$ has an induced $V$-module structure.
		\end{itemize}
		\item\label{support_H0} If $\cM$ is a coherent sheaf, then \[ \Supp \cH^0(\FM_A(\cM)) = \bigset{\alpha \in \bighat{A}}{ H^0(A, \cM \otimes \cP_{\alpha}) \neq 0}. \]
	\end{enumerate}
\end{thm}
\begin{proof}
	The statements \autoref{itm:equiv_cat} and \autoref{itm:behaviour_pushforwards_and_pullbacks} follow respectively from \cite[Theorem 3.2 and 5.1]{Schnell_Fourier_Mukai_transform_made_easy}, and \autoref{support_H0} is \cite[Theorem 2.2.8.(g)]{Baudin_Positive_characteristic_generic_vanishing_theory}.
\end{proof}

One of the main statements in positive characteristic generic vanishing theory says that the Fourier--Mukai transform induces an equivalence of categories between Cartier crystals on $A$ and $V$--crystals on $\bighat{A}$, see \cite[Theorem 3.2.1]{Baudin_Positive_characteristic_generic_vanishing_theory}. A relevant cohomological consequence of this general fact to us is the following statement:

%
%
%
\begin{lem}\label{SCSL for H^i = SCSL for Tor^i}
	Let $\cM$ be a Cartier module on $A$, and consider the $V$--module $\cN \coloneqq \cH^0(\FM_A(\cN))$. Then for all $i \geq 0$, \[ \varprojlim H^i(A, F^{e}_*\cM)  \cong \left(\colim \Tor_i(V^{e, *}\cN, k(0))\right)^{\vee}. \] In particular, \[ \chi_{ss}(A, \cM) = \chi_{ss}(\cN). \]
\end{lem}
\begin{proof}
	See \cite[Lemma 3.3.7]{Baudin_Positive_characteristic_generic_vanishing_theory}.
\end{proof}

Thanks to \autoref{SCSL for H^i = SCSL for Tor^i}, our approach to prove \autoref{intro_prop_euler_char} will be to study the semistable Euler characteristic of $V$--modules.

\subsection{Semistable Euler characteristic of $V$-modules}

Recall that a module is said to be \emph{torsion} if it is zero on some dense open subset. Our goal here is to show the following:

\begin{prop}\label{euler_char_V_modules}
	Assume that $A$ is ordinary. Then for any $V$--module $\cN$ on $\bighat{A}$, we have \[ \chi_{ss}(\cN) \geq 0. \] If $\cN$ is in addition torsion, then \[ \chi_{ss}(\cN) = 0. \]
\end{prop}

Throughout, assume that $A$ is ordinary, and let $g \coloneqq \dim(A)$. The idea will be to work at the completion $\cO_{\bighat{A}, 0} \cong k\llbracket x_1, \dots, x_g \rrbracket \eqqcolon R$. Let $\fm$ denote its maximal ideal. The following lemma will make the situation much easier:

\begin{lemma}\label{nice_form_for_V}
	There exists $y_1, \dots, y_g \in R$ such that $\fm = (y_1, \dots, y_g)$ and $V(y_i) = y_i$.
\end{lemma}
\begin{rem}
	Recall that $V|_k = (\cdot)^{1/p}$ by construction, so the lemma above does \emph{not} imply that $V = \id_R$ (which would be very wrong).
\end{rem}
\begin{proof}
	Let us first construct the $y_i$'s up to first order. Write \[ V(x_i) = \sum_j a_{ij} x_j \mod \fm^2, \] with $A \coloneqq (a_{ji}) \in \GL_g(k)$ ($V$ is étale by \autoref{equivalence_ordinary}), and let $\lambda = (\lambda_1, \dots, \lambda_g)^T \in k^g$. Then modulo $\fm^2$, we have \begin{equation}\label{equation_fixed_points}
		V\left( \sum_i \lambda_ix_i \right) \longexpl{=}{$V|_k = (\cdot)^{1/p}$ by construction} \sum_j \left(\sum_i \lambda_i^{1/p}a_{ij}\right) x_j.
	\end{equation} Hence, to find $g$ generators of $\fm$ fixed by $V$ up to second order, we have to find a basis $\lambda(1), \dots, \lambda(g)$ of $k^g$ such that for all $l$, \[ A \cdot \lambda(l)^{1/p} = \lambda(l), \] where for a vector $v = (v_1, \dots, v_g)^T \in k^g$, we write $v^{1/p} = (v_1^{1/p}, \dots, v_g^{1/p})$. Indeed, then the elements $\sum_i y_i = \lambda(l)_ix_i$ will generate $\fm$ and be fixed by $V$ modulo $\fm^2$ by \autoref{equation_fixed_points}.

	 Up to the change of variables $\mu(l) \coloneqq \lambda(l)^{p}$ and setting $B \coloneqq A^{-1}$, we want a basis of $k^g$ of solutions of the system \[ B\cdot \mu(l)^{p} = \mu(l). \] Note that the additive map \[ \begin{tikzcd}[row sep = tiny]
		k^g \arrow[rr, "\phi"]        &  & k^g            \\
		v \arrow[rr, maps to] &  & B\cdot v^{p}. \end{tikzcd} \] 
	is $p$--linear and injective. By \cite[Corollary p.143]{Mumford_Abelian_Varieties}, there is a basis of $k^g$ consisting of fixed points of $\phi$. We then conclude that there exist $y_1, \dots, y_g \in R$ such that $\fm = (y_1, \dots, y_g)$ and $V(y_i) = y_i$ modulo $\fm^2$. \\
	
	Let us construct a new generating set that is fixed modulo $\fm^3$. Write \[ V(y_i) = y_i + \sum_{s,t}b^i_{st}y_sy_t \mod \fm^3 \] with $b^i_{st} \in k$, and let $c^i_{st}$ be a solution of the equation \[ b^i_{st} + (c^i_{st})^{1/p}  = c^i_{st}. \] (such a solution exists since $k$ is algebraically closed). Then one checks immediately that \[ V\left(y_i + \sum_{s, t}c^i_{st}y_sy_t\right) = y_i + \sum_{s, t}c^i_{st}y_sy_t \mod \fm^3, \] i.e. we have our generating set fixed modulo $\fm^3$. Continuing indefinitely like this concludes the proof.
\end{proof}

From now on, assume that $V(x_i) = x_i$ for all $i$ (we may do so by \autoref{nice_form_for_V}). Since $V$ is an endomorphism of $R$, there is an straightforward notion of $V$--module on $R$. We will also use \autoref{not_euler_char_V_mods} in this setting. Furthermore, note that by assumption, $V$ also restrict to a morphism $V \colon R/x_i \to R/x_i$. When $N$ is a $V$--module on $R/x_i$ (and hence it is also a $V$-module on $R$), we will write $\chi_{ss}^{R/x_i}(N)$ (resp. $\chi_{ss}^R(N)$) for the quantity $\chi_{ss}$ computed on $R/x_i$ (resp. on $R$).

\begin{lemma}\label{euler_char_torsion_V_modules}
	Let $N$ be a $V$--module on $R = k\llbracket x_1, \dots, x_g \rrbracket$ such that $N$ is $x_i^e$--torsion for some $e > 0$ and $i \geq 0$. Then $\chi_{ss}(N) = 0$.
\end{lemma}
\begin{proof}
	For all $j$, let $N_j \inc N$ be the submodule of $x_i^j$-torsion elements. Then all these $N_j$'s are $V$-submodules, and the successive quotients of the filtration $0 \inc N_1 \inc N_2 \inc \dots \inc N$ are all $x_i$--torsion $V$-modules. By additive of the Euler characteristic, we are reduced to the case where $N$ is $x_j$--torsion. Note that $R$ is naturally a $V$--module, and since $V(x_i) = x_i$, we have an exact sequence of $V$--modules \[ \begin{tikzcd}
		0 \arrow[rr] &  & R \arrow[rr, "x_i"] &  & R \arrow[rr] &  & R/x_i \arrow[rr] &  & {0,}
	\end{tikzcd} \]
	(hence $R/x_i$ is also a $V$-module). This shows that as $V$--modules on $R/x_i$, \[ \Tor^R_0(N, R/x_i) \cong \Tor^R_1(N, R/x_i) \cong N. \] A straightforward computation with the Leray spectral sequence \[ \Tor^{R/x_i}_a(\Tor^R_b(N, R/x_i), k) \implies \Tor^R_{a + b}(N, k) \] shows that \[ \chi^R_{ss}(N) = \chi^{R/x_i}_{ss}(N) - \chi^{R/x_i}_{ss}(N) = 0. \qedhere \]
\end{proof}

\begin{proof}[Proof of \autoref{euler_char_V_modules}]
	As already stated, it is enough to show the result for  $V$--modules on $R = k\llbracket x_1, \dots, x_g\rrbracket \cong \cO_{\bighat{A}, 0}$, and assume by \autoref{nice_form_for_V} that $V(x_i) = x_i$. 
	
	Let $N$ be a $V$--module on $R$. Then we will show that $\chi_{ss}(N) \geq 0$, and if $N$ is torsion, that $\chi_{ss}(N) = 0$. Let us show the result by induction on $g$. If $g = 0$, the result is trivial, so assume that $g > 0$.
	
	Let $N' \subseteq N$ denote the $V$--submodule of $N$ consisting of $x_g^e$--torsion elements for some $e > 0$. Then we have a short exact sequence \[ \begin{tikzcd}
		0 \arrow[rr] &  & N' \arrow[rr] &  & N \arrow[rr] &  & N'' \arrow[rr] &  & {0,}
	\end{tikzcd} \] so by \autoref{euler_char_torsion_V_modules} and additivity of $\chi_{ss}$, it is enough to show the result for $N''$, which is $x_g$--torsion free. Then $\Tor_i^R(N'', R/x_g) = 0$ for all $i > 0$, so by the Leray spectral sequence \[ \Tor_j^{R/x_g}\left(\Tor_i^R(N'', R/x_g), k\right) \implies \Tor_{i + j}^R(N'', k), \] we have \[ \chi^R_{ss}(N'') = \chi^{R/x_g}_{ss}(N''/x_gN''). \] Hence, we can conclude the proof by induction (note that if $N$ was torsion to begin with, then also $N''/x_gN''$ would be torsion in $R/x_g$ too).
\end{proof}

\begin{cor}\label{Euler_char_Cartier_modules_ordinary}
	Let $\cM$ be a Cartier module on an ordinary abelian variety $A$. Then $\chi_{ss}(A, \cM) \geq 0$. Furthermore, if there exists a Cartier module $\cM' \sim_C \cM$ and $L \in \Pic^0(A)$ such that $H^0(A, \cM' \otimes L) = 0$, then $\chi_{ss}(A, \cM) = 0$.
\end{cor}
\begin{proof}
	The first statement follows from \autoref{euler_char_V_modules} and \autoref{SCSL for H^i = SCSL for Tor^i}. Assume now that there exists a Cartier module $\cM'$ such that $\cM' \sim_C \cM$ and that $H^0(A, \cM' \otimes L) = 0$ for some $L \in \Pic^0(A)$. This second hypothesis implies that $\Supp \cH^0(\FM_A(\cM')) \neq \bighat{A}$ by \autoref{properties_Fourier_Mukai_transform}.\autoref{support_H0}, so $\chi_{ss}(A, \cM') = 0$ by \autoref{euler_char_V_modules} and \autoref{SCSL for H^i = SCSL for Tor^i}. Finally, since $\cM' \sim_C \cM$, we obtain that $\cN^{\perf} \cong \cM^{\perf}$ by \autoref{Gabber_finiteness}.\autoref{itm:nil_iso_iff_perf_iso} and \cite[Lemma 3.5.(b)]{Blickle_Bockle_Cartier_modules_finiteness_results}, and hence $\chi_{ss}(A, \cM') = \chi_{ss}(A, \cM)$.
\end{proof}

%

%

\section{Proof of the main results}

Let us recall the existence of Albanese varieties that we need in our generality. The reference that we use (\cite{Schroer_Laurent_Para_Abelian_varieties_and_Albanese_maps}) deals with the more general notion of para--abelian varieties. Since $k$ is algebraically closed, a para--abelian abelian variety is the same as an abelian variety, see \cite[First paragraph in Section 4 and Proposition 4.3]{Schroer_Laurent_Para_Abelian_varieties_and_Albanese_maps}
. 
\begin{sthm}\label{existence_albanese}
	Let $X$ be a proper $k$--scheme such that $H^0(X, \cO_X) = k$. Then there exists a unique morphism $\alb_X \colon X \to \Alb(X)$ to an abelian variety $\Alb(X)$ such that for any other morphism $g \colon X \to Q$ to an abelian variety, there exists a unique morphism $h \colon \Alb(X) \to Q$ such that $h \circ \alb_X = g$.
	
	Furthermore, the pullback of line bundles map $\Pic^0_{\Alb(X)/k} \to \Pic^0_{X/k}$ is a closed immersion, and the pullback of sections map $H^1(\Alb(X), \cO_{\Alb(X)}) \to H^1(X, \cO_X)$ is injective.
\end{sthm}
\begin{proof}
	Everything, except the statement about $H^1$--groups, follows from \cite[Definition 8.1 and Theorem 10.2]{Schroer_Laurent_Para_Abelian_varieties_and_Albanese_maps}. Since the tangent space at $\cO_X$ of $\Pic^0_{X/k}$ is exactly $H^1(X, \cO_X)$, and similarly for $\Alb(X)$ instead of $X$ (see for example \cite[Theorem 9.5.11]{Kleiman_The_Picard_Scheme} and \cite[Theorem 2.1]{Schroer_Laurent_Para_Abelian_varieties_and_Albanese_maps}), the statement follows.
\end{proof}
Before stating our main theorem and use the weak Grauert--Riemenschneider sheaf, let us start with the key lemma that will allow us to deal with the case of varieties not of general type. In characteristic zero, this is an immediate consequence of \cite[Theorem 1]{Chen_Hacon_Pluricanonical_maps_of_vars_of_max_alb_dim}. We will adapt their strategy in this very particular case.

\begin{slem}\label{not_of_gen_type_implies_locus_not_full}
	Let $X$ be a normal proper variety with canonical singularities, and let $a \colon X \to A$ be a generically finite morphism to an abelian variety. If $X$ is not of general type, then there exists $L \in \Pic^0(A)$ such that $H^0(X, \omega_X \otimes a^*L) = 0$.
\end{slem}
\begin{proof}
	We need to show that $V^0(a_*\omega_X) \neq \bighat{A}$. Since this can be done after any base change to a field extension of $k$, we may assume that $k$ is uncountable.
	
	Let us show the contrapositive, so assume that $H^0(X, \omega_X \otimes a^*L) \neq 0$ for all $L \in \Pic^0(A)$. We will show that $K_X$ is big. Since $K_X$ is $\bQ$--Cartier by assumption, we can take its associated Iitaka fibration, see e.g. \cite[Theorem 2.1.3.3]{Lazarsfeld_Positivity_in_algebraic_geometry_I} (since this reference requires projectivity, we first apply Chow's lemma and take the Iitaka fibration of the pullback of $K_X$ to this normal projective model). In other words, we have a diagram \[ \begin{tikzcd}
		Y \arrow[rr, "\pi"] \arrow[d, "g"'] &  & X \arrow[rr, "a"] &  & A, \\
		Z                                   &  &                        &  &        
	\end{tikzcd} \] where 
	
	\begin{itemize}
		\item $Y$ is normal;
		\item $\pi$ is birational and projective;
		\item $\dim(Z) = \kappa(X, K_X)$;
		\item $\kappa(Y_{\eta}, K_{Y_{\eta}}) = 0$, where $\eta$ denotes the generic point of $Z$ (here, we use that $X$ has canonical singularities, see e.g. \cite[Lemma 6.2]{Baudin_Effective_Characterization_of_ordinary_abelian_varieties}.
	\end{itemize}
	
	Since $Y$ has maximal Albanese dimension, we know by \cite[Theorem 4.3]{Baudin_Positive_characteristic_generic_vanishing_theory} that $H^0(Y, \omega_Y) \neq 0$, so in particular $H^0(Y_{\eta}, \omega_{Y_{\eta}}) \neq 0$.
	
	Let $h \coloneqq a \circ \pi$, and let $m > 0$ be an integer such that $mK_X$ is Cartier. Since we also know that $H^0(X, \omega_X) \neq 0$ by \emph{loc. cit.}, we deduce that for all $L \in \Pic^0(A)$, we have $H^0(X, \cO_X(mK_X) \otimes a^*L) \neq 0$. Since $X$ has canonical singularities, the $\bZ$--divisor $mK_Y - \pi^*(mK_X)$ is effective, so also $H^0(Y, \cO_Y(mK_Y) \otimes h^*L) \neq 0$ for all $L \in \Pic^0(A)$. This shows in particular that \[ H^0(Y_{\eta}, \cO_{Y_{\eta}}(mK_{\eta}) \otimes h^*L) \neq 0 \] for all $L \in \Pic^0(A)$. Since $\kappa(Y_{\eta}, K_{Y_{\eta}}) = 0$, this forces $h^* \colon \Pic^0(A) \to \Pic(Y_{\eta})$ to be the trivial map  (this is standard in generic vanishing theory, see e.g. the argument of \cite[Proposition 2.1]{Ein_Lazarsfeld_Singularities_of_theta_divisors_and_the_birational_geometry_of_irregular_varieties}). 
	
	Let us deduce that for a very general closed fiber $G$ of $g$, the pullback map $\Pic^0_{A/k} \to \Pic^0_{G/k}$ is constant. Pick $L \in \Pic^0(A)$. Since $h^*L$ is trivial on the generic fiber, we know by \cite[Corollaire 9.4.6]{EGA_IV.3}, it is also trivial when restricted to a general fiber of $g$. In particular, we deduce that $h^*L|_G \cong \cO_G$ for all \emph{torsion} line bundle $L \in \Pic^0(A)$, since $G$ is very general and there are only countably many torsion line bundles in $\Pic^0(A)$. Given that the set of torsion elements in $\Pic^0(A)$ is dense in the scheme $\Pic^0_{A/k}$, we deduce that the induced morphism $\Pic^0_{A/k} \to \Pic^0_{G/k}$ is constant. By \autoref{existence_albanese}, this shows that the induced morphism $\Alb(G) \to A$ is constant. Since $h|_G \colon G \to A$ factors through $\Alb(G)$ by the universal property, we deduce that $h|_G$ is constant. Since it must also be generically finite ($h$ is generically finite), we deduce that $\dim(G) = 0$.  In other words, $\dim(X) = \kappa(X, K_X)$, so $K_X$ is big. \qedhere
	
%
%
\end{proof}

Now, let us prove our main theorem.

\begin{sthm}\label{main_thm_ss}
	Let $X$ be a smooth proper variety admitting a generically finite morphism to an ordinary abelian variety $A$. Then \[ \chi_{ss}(X, \omega_X) \geq 0. \] If $X$ is furthermore not of general type, then  \[ \chi_{ss}(X, \omega_X) = 0. \]
\end{sthm}
\begin{proof}
	Let $a \colon X \to A$ be a generically finite morphism to an ordinary abelian variety. Then by \autoref{weak_GR_sheaf}, we know that \[ \chi_{ss}(X, \omega_X) = \chi_{ss}(A, \omega_{a, \GR}). \]
	Since $\omega_{a, \GR}$ is a Cartier crystal and $A$ is ordinary, we deduce that $\chi_{ss}(A, \omega_{a, \GR}) \geq 0$ by \autoref{Euler_char_Cartier_modules_ordinary}, so the first part is proven.
	
	Let us now assume that $X$ is not of general type. Note that combining \autoref{rem_we_may_assume_sub_Cartier_mod} and \autoref{Euler_char_Cartier_modules_ordinary}, it is enough to show that $H^0(A, a_*\omega_X \otimes L) = 0$ for some $L \in \Pic^0(X)$. This is exactly \autoref{not_of_gen_type_implies_locus_not_full}.
\end{proof}

We find the following application rather interesting:

\begin{sprop}\label{suprising_GR_vanishing}
	Let $X$ be a smooth proper threefold, and let $a \colon X \to A$ be a generically finite morphism to an ordinary abelian variety $A$. If $K_X$ is not big, then Grauert--Riemenschneider holds up to nilpotence for $a$, i.e. \[ R^ia_*\omega_X \sim_C 0 \] for all $i > 0$.
\end{sprop}
\begin{proof}
	Since $\dim(X) = 3$, we know by \cite[Proposition 2.6.1]{Hacon_Pat_GV_Characterization_Ordinary_AV} (see the arXiv version) that $R^ia_*\omega_X = 0$ for all $i > 1$ and that $R^1a_*\omega_X$ is at most supported at finitely many points. Note also that the same argument as in the proof of \autoref{main_thm_ss} shows that $\chi_{ss}(A, a_*\omega_X) = 0$, so we obtain that
	\[  0 \expl{=}{\autoref{main_thm_ss}} \chi_{ss}(X, \omega_X) = \chi_{ss}(A, a_*\omega_X) - \chi_{ss}(A, R^1a_*\omega_X) = - \chi_{ss}(A, R^1a_*\omega_X).  \]
	Since $R^1a_*\omega_X$ is supported at closed points, we have have $\chi_{ss}(A, R^1a_*\omega_X) = h^0_{ss}(A, R^1a_*\omega_X)$ so we deduce that \[ h^0_{ss}(A, R^1a_*\omega_X) = 0. \] Again, since $R^1a_*\omega_X$ is supported at points, it follows that $R^1a_*\omega_X \sim_C 0$.
\end{proof}

It would be interesting to study whether in this case, we actually have that $R^1a_*\omega_X = 0$ and that $a_*\omega_X$ is a GV--sheaf.

\begin{example_sec}\label{example_non_GR_van}
	If $K_X$ is big or if $\dim(X) \geq 4$, the statement of \autoref{suprising_GR_vanishing} is false in general. Indeed, by \cite[Corollary 5.4]{Baudin_Bernasconi_Kawakami_Frobenius_GR_fails}, there exists a smooth projective threefold $Y$ with a generically finite map to an abelian variety $A$, such that Grauert--Riemenschneider vanishing fails (even up to nilpotence). Note that by construction (see the proof of \cite[Proposition 3.13]{Hacon_Kovacs_GV_fails_in_pos_char}), we may choose $A$ to be an ordinary abelian variety. The variety $Y$ must then necessarily be of general type by \autoref{suprising_GR_vanishing}.
	
	To obtain higher--dimensional counterexamples, simply take $Y \times B$ for $B$ any ordinary abelian variety.

\end{example_sec}

\begin{sdefn}
	A proper variety $X$ is said to be \emph{weakly ordinary} if for all $i \geq 0$, the natural map $H^i(X, \cO_X) \to H^i(X, F_*\cO_X)$ is an isomorphism.
\end{sdefn} 

\begin{scor}\label{main_thm_ordinary}
	Let $X$ be a smooth proper weakly ordinary variety of maximal Albanese dimension. Then \[ \chi(X, \omega_X) \geq 0. \] If $X$ is in addition not of general type, then \[ \chi(X, \omega_X) = 0. \]
\end{scor}
\begin{proof}
	By definition of weak ordinarity and Serre duality, we know that $\chi_{ss}(X, \omega_X) = \chi(X, \omega_X)$. By \autoref{main_thm_ss}, it is then enough to show that $\Alb(X)$ is ordinary. Given that the pullback map $H^1(\Alb(X), \cO_{\Alb(X)}) \to H^1(X, \cO_X)$ is injective (\autoref{existence_albanese}) and that the Frobenius action on $H^1(X, \cO_X)$ is bijective, there cannot be any nilpotent element in $H^1(\Alb(X), \cO_{\Alb(X)})$.
	By definition, $\Alb(X)$ is then ordinary.
\end{proof}

Given a Cartier module $\cM$, we denote by $V^0_{\injj}(\cM)$ the support of $\cH^0\FM_A(\cM)_{\injj}$ (see \autoref{canonical_injective_V_module} for the $(\cdot)_{\injj}$ notation).

\begin{sthm}\label{euler_char_zero_implies_alb_fibered_by_ab_vars}
	Let $X$ be a smooth proper weakly ordinary variety of maximal Albanese dimension, and assume that $\chi(X, \omega_X) = 0$. Then $\alb_X(X)$ is fibered by ordinary abelian varieties.
\end{sthm}

In the proof of \autoref{euler_char_zero_implies_alb_fibered_by_ab_vars}, we will need the following result:

\begin{sprop}\label{fibering_alb_image}
	Let $Y$ be a normal variety of maximal Albanese dimension with Albanese morphism $a \colon Y \to A$, and assume that $V^0_{\injj}(a_*\omega) \neq \bighat{A}$ for some non--zero Cartier module $\omega \inc \omega_Y$. Then $a(Y)$ is fibered by abelian varieties.
\end{sprop} 
\begin{proof}
	This is \cite[Proposition 4.2.8]{Baudin_Effective_Characterization_of_ordinary_abelian_varieties}.
\end{proof}

\begin{proof}[Proof of \autoref{euler_char_zero_implies_alb_fibered_by_ab_vars}]
	 To lighten notations, let $a \colon X \to A$ denote the Albanese morphism of $X$. Throughout this proof, all points will be closed points, and we will use \autoref{rem_we_may_assume_sub_Cartier_mod} without further mention. Consider the Stein factorization $X \xrightarrow{\pi} Y \xrightarrow{b} A$ of $a$. Since $\pi_*\omega_{\pi, \GR} \inc \pi_*\omega_X \inc \omega_Y$, it is enough to show that $V^0_{\injj}(b_*\omega_{\pi, \GR}) \neq \bighat{A}$ by \autoref{fibering_alb_image}. By \cite[Lemma 3.2.9]{Baudin_Effective_Characterization_of_ordinary_abelian_varieties}, we may assume that $k$ is uncountable. We then obtain by \cite[Theorem 3.3.5.(e) and Proposition 3.3.17.(c)]{Baudin_Positive_characteristic_generic_vanishing_theory} that showing that $V^0_{\injj}(b_*\omega_{\pi, \GR}) \neq \bighat{A}$ is equivalent to showing that for very general $\alpha \in \bighat{A}$, \[ H^0(A, b_*\omega_{\pi, \GR}^{\perf} \otimes \cP_{\alpha}) = 0. \] Note that \[ b_*\omega_{\pi, \GR}^{\perf} \cong b_*\left((\pi_*W\omega_X)/p\right) \expl{\cong}{$b$ is finite} (a_*W\omega_X)/p \cong \omega_{a, \GR}^{\perf},  \] so we have to show that  \[ H^0(A, \omega_{a, \GR}^{\perf} \otimes \cP_{\alpha}) = 0 \] for very general $\alpha \in \bighat{A}$. Consider the diagram 
	 \[ \begin{tikzcd}
		& X \times \bighat{A} \arrow[ld, "\pi"'] \arrow[rd, "q"] &            \\
		X &                                           & \bighat{A},
	\end{tikzcd} \] where both maps denote the projections, and let $\cL \coloneqq (a, \id_{\bighat{A}})^*\cP \in \Pic(X \times \bighat{A})$. Note that by \stacksproj{08IB}, for all $\alpha \in \bighat{A}$ and $\cM \in \Coh(X)$, we have \[ Rq_*(\pi^*\cM \otimes \cL) \otimes^L k(\alpha) \cong \RGamma(X, \cM \otimes \cL_{\alpha}). \] In particular, given that the map $Rq_*(\pi^*F_*\omega_X \otimes \cL) \to Rq_*(\pi^*\omega_X \otimes \cL)$ induced by the Cartier operator $F_*\omega_X \to \omega_X$ is an isomorphism when applying $ - \otimes^L k(0)$ ($X$ is weakly ordinary), it must be an isomorphism over some open $0 \in U \subseteq \bighat{A}$. Let $V \coloneqq \bigcap_{e > 0} p^e(U) \subseteq \bighat{A}$ (which may not be open anymore). Then all the maps \[ H^i(X, \omega_X^{\perf} \otimes \cL_{\alpha}) \to H^i(X, \omega_X \otimes \cL_{\alpha}) \] are isomorphisms whenever $\alpha \in V$ by definition and the projection formula. Hence, in this case, we obtain that \[ 0 = \chi(X, \omega_X) = \chi(X, \omega_X \otimes \cL_{\alpha}) = \chi(X, \omega_X^{\perf} \otimes \cL_{\alpha}) \expl{=}{\autoref{weak_GR_sheaf_on_steroids} and \autoref{cohom_perfection_twist_fg}} \chi(A, \omega_{a, \GR}^{\perf} \otimes \cP_{\alpha}). \]
	However, we know that for $\alpha \in \bighat{A}$ very general, $H^i(A, \omega_{a, \GR}^{\perf} \otimes \cP_{\alpha}) = 0$ for all $i > 0$ by \cite[Theorem 3.3.5.(e)]{Baudin_Positive_characteristic_generic_vanishing_theory} (or \cite[Theorem 1.1]{Hacon_Pat_GV_Geom_Theta_Divs}), so $H^0(A, \omega_{a, \GR}^{\perf} \otimes \cP_{\alpha}) = 0$ too if in addition $\alpha \in V$. In other words, we have proven that $H^0(A, \omega_{a, \GR}^{\perf} \otimes \cP_{\alpha}) = 0$ for $\alpha$ very general. \qedhere
%
%
%
%
%
\end{proof}

\bibliographystyle{alpha}
\bibliography{bibliography}

\Addresses

\end{document}